\documentclass[a4paper,12pt]{article}
\usepackage{array,amsfonts,amsmath,graphicx,mathrsfs,amssymb,amsthm,latexsym}
\usepackage[latin1]{inputenc}
\def \msp {\vspace{-1ex}}

\newtheorem{thm}{Theorem}[section]
\newtheorem{lemma}[thm]{Lemma}

\def \cF {{\cal F}}

\def \cH {{\cal H}}

\def \cP {{\cal P}}

\begin{document}
\title{\vspace{-10ex} ~~ \\
Uniformly resolvable $(C_4, K_{1,3})$-designs of order $v$ and index 2}

\author {Mario Gionfriddo \thanks{Research supported by MIUR, Italy and CNR-GNSAGA}\\
\small Dipartimento di Matematica e Informatica \\
\small Universit\`a di Catania \\
\small Catania\\
\small Italia \\
{\small \tt gionfriddo@dmi.unict.it} \\  Selda K\"{u}\c{c}\"{u}k\c{c}\.{i}f\c{c}\.{i}
\thanks{Corresponding author, phone: $+90 \ 212 \ 338 \ 1523$, fax: $+90 \ 212 \ 338 \ 1559$}
\thanks{Research supported by Scientific and Technological Research Council of Turkey Grant Number: 114F505}\\
\small \msp Department of Mathematics \\
\small \msp Ko\c{c} University \\
\small \msp Istanbul\\
\small Turkey\\
{\small \vspace{1ex} \tt skucukcifci@ku.edu.tr} \\  Salvatore Milici
\thanks{Supported b by C. N. R. (G. N. S. A. G. A.), Italy}\\
\small \msp Dipartimento di Matematica e Informatica \\
\small \msp Universit\`a di Catania \\
\small \msp Catania\\
\small Italia\\
{\small \vspace{1ex} \tt milici@dmi.unict.it}
\\  E. \c{S}ule Yaz\i c\i \\
\small \msp Department of Mathematics \\
\small \msp Ko\c{c} University \\
\small \msp Istanbul\\
\small Turkey}

\maketitle

\begin{abstract}

In this paper we consider the uniformly resolvable decompositions of
the complete graph $2K_v$ into subgraphs where each resolution class
contains only blocks isomorphic to the same graph. We completely
determine the spectrum for the cases  in which all the resolution
classes are either $C_4$ or $K_{1,3}$.

\end{abstract}

\vbox{\small
\vspace{5 mm}
\noindent \textbf{AMS Subject classification:} $05C51$, $05C38$, $05C70$.\\
\textbf{Keywords:} Resolvable graph decomposition; uniform resolutions;
3-stars; $4$-cycles.
 }

\section{Introduction
 }\label{introduzione}

Given a collection of graphs $\cH$, an {\em $\cH-$decomposition} of
a graph $G$ is a decomposition of the edges of $G$ into isomorphic
copies of graphs in $\cH$. The copies of $H\in\cH$ in the
decomposition are called {\em blocks}. Such a decomposition is
called {\em resolvable} if it is possible to partition the blocks
into {\em classes} $\cP_i$ such that every point of $G$ appears
exactly once in some block of each $\cP_i$ \cite{G}.

A resolvable $\cH-$decomposition of $G$ is sometimes also referred
to as an {\em $\cH-$factorization of $G$} and a resolution class is
called an {\em $\cH$-factor of $G$}. The case where $\cH$ is a
single edge ($K_2$) is known as a {\em $1$-factorization of $G$} and
it is well known to exist for $G=K_v$ if and only if $v$ is even. A
single class of a $1$-factorization, a pairing of all points, is
also known as a {\em $1$-factor\/} or a {\em perfect matching}.

In many cases we wish to impose further constraints on
the classes of an $\cH$-decomposition.
 For example, a class is called {\em uniform\/} if every
block of the class is isomorphic to the same graph in $\cH$.
Uniformly resolvable decompositions of $K_v$ have also been studied
in \cite{DQS}, \cite{DLD}, \cite{GM} -- \cite{GM2}, \cite{KMT} --
\cite{S3}.

In this paper we study the existence of a uniformly resolvable
decomposition of $2K_v$  having the following type:
\begin{quote}
 $r$ classes containing only copies of 4-cycles
 and $s$ classes containing only copies of 3-stars.
\end{quote}
We will use the notation $(C_4, K_{1,3})$-URD$(v,2;r,s)$ for such
a uniformly resolvable decomposition of $2K_{v}$.
Let now

\begin{quote}
  $J(2K_v;C_4,K_{1,3}) = \{(r,s)$ : there exists a uniformly resolvable decomposition of
$2K_v$ into $r$ classes containing only copies of $C_4$ and $s$
classes containing only copies of $K_{1,3}\}$.
\end{quote}

For $v\geq4$, divisible by 4, define $I(v)$ according to the following table:

\vspace{4 mm}

 \begin{minipage}[t]{\textwidth}
\begin{center}
\begin{tabular}{|c|c|}
\hline
  $v $ &  {$I(v)$}
\\
\hline
$0 \pmod{12}$ & $\{(v-1-3x, 4x), x=0,1,\ldots,\frac{v-3}{3}\} $\\
$4 \pmod{12}$ & $\{(v-1-3x, 4x), x=0,1,\ldots,\frac{v-1}{3}\}$\\
$8 \pmod{12}$ & $\{(v-1-3x, 4x), x=0,1,\ldots,\frac{v-2}{3}\}$\\

\hline
 \end{tabular}

\bigskip

Table 1: The set $I(v)$.

 \end{center}\end{minipage}

\vspace{4 mm}

In this paper we completely solve the spectrum problem for such
 systems; that is, characterize the existence of uniformly resolvable
decompositions of $2K_{v}$  into $r$ classes of 4-cycles
and $s$ classes of 3-stars, by proving the following result:\\

\noindent \textbf{Main Theorem.} { \em
 For every integer\/ $v\geq4$, divisible by\/ $4$,
 the set\/ $J(2K_v;C_4$, $K_{1,3})$ is identical to the set\/ $I(v)$
  given in Table $1$.}\vspace{0.2in}

Now let us introduce some useful definitions, notations, results and
discuss constructions we will use in proving the main theorem. For
missing terms or results that are not explicitly explained in the
paper, the reader is referred to \cite{CD} and its online updates.
For some results below, we also cite this handbook instead of the
 original papers.

 For any four vertices $a_1,a_2,a_3,a_4$,
 let the {\em $3$-star\/}, $K_{1,3}$,
 be the simple graph with the vertex set
$\{a_1,a_2,a_3,a_4\}$ and the edge set $\{\{a_1,a_2\}, \{a_1,a_3\},
\{a_1,a_4\}\}$ and the 4-cycle $C_4$ be the simple graph with the
vertex set $\{a_1,a_2,a_3,a_4\}$ and the edge set $\{\{a_1,a_2\},
\{a_2,a_3\}, \{a_3,a_4\}, \{a_4,a_1\}\}$. In what follows, we will
denote the 3-star by  $(a_1;a_2,a_3,a_4)$ and the 4-cycle by
$(a_1,a_2,a_3,a_4)$, $(a_4,a_3,a_2,a_1)$ or any cyclic shift of
these.

 A resolvable $\cH$-decomposition of the
complete multipartite graph with $u$ parts each of size $g$  is
known as a resolvable group divisible design $\cH$-RGDD of type
$g^u$, the parts of size $g$ are called the groups of the design.
When $\cH = K_n$ we will call it an $n$-RGDD.

A $(C_4, K_{1,3})$-URGDD $(2;r,s)$ of type $g^u$  is a uniformly resolvable
decomposition of the complete multipartite graph of index 2 with $u$ parts
each of size $g$ into $r$ classes containing only copies of 4-cycles
and $s$ classes containing only copies of 3-stars.

If the blocks of an $\cH$-GDD of type ${g^u}$ can be partitioned
into partial parallel classes, each of which contain all points
except those of one group, we refer to the decomposition as a {\em
frame}. When $\cH = K_n$ we will call it an $n$-{\em frame} and it
can be deduced that the number of partial parallel classes missing a
specified group $G$ is $\frac{|G|}{n-1}$.

An incomplete resolvable $(C_4, K_{1,3})$-decomposition of $2K_v$
with a hole of size $h$ is a $(C_4, K_{1,3})$-decomposition of
$2K_{v+h}- 2K_h$ in which there are two types of classes, full
classes and partial classes which cover every point except those in
the hole (the points of $2K_h$ are referred to as the hole).
Specifically a $(C_4, K_{1,3})$-IURD$(2K_{v+h}- 2K_h; [r,s],
[\bar{r}, \bar{s}])$ is a uniformly resolvable $(C_4,
K_{1,3})$-decomposition of $2K_{v+h}-2K_h$ with $r$ partial classes
of $4$-cycles which cover only the points not in the hole, $s$
partial classes of $3$-stars which cover only the points not in the
hole, $\bar{r}$ full classes of $4$-cycles which cover every point
of $2K_{v+h}$ and $\bar{s}$ full classes of $3$-stars which cover
every point of $2K_{v+h}$.

We also need the following definitions. Let $(s_1, t_1)$ and $(s_2,
t_2)$ be two pairs of non-negative integers. Define $(s_1, t_1)
+(s_2, t_2)=(s_1+s_2, t_1+t_2)$. If $X$ and $Y$ are two sets of
pairs of non-negative integers, then $X+Y$ denotes the set $\{(s_1,
t_1) +(s_2, t_2) : (s_1, t_1)\in X, (s_2, t_2) \in Y \}$. If $X$ is
a set of pairs of non-negative integers and $h$ is a positive
integer, then  $h * X$ denotes the set of all pairs of non-negative
integers which can be obtained by adding any $h$ elements of $X$
together (repetitions of elements of $X$ are allowed).

Similarly to what was done in \cite{KLMT}, the following results can
be proven.

\begin{lemma}
\label{lemmaP1} If there exists a\/ $(C_4, K_{1,3})$-URD$(v,2;r,s)$ of\/
$2K_v$ with\/ $r>0$ and\/ $s>0$
 then\/ $v\equiv 0\pmod{4}$ and\/ $(r,s)\in I(v)$.
\end{lemma}

The following  lemma will be very useful in proving the main results
of this paper.

\begin{lemma}
\label{lemmaP2} Let\/ $v$, $g$, and\/ $u$   be non-negative integers
such that\/ $v=gu$. If there exists
\begin{itemize}
\item
[$(1)$] a\/ $4$-RGDD of type\/ $g^{u}$;

\item
[$(2)$] a $(C_4, K_{1,3})$-URD$(4,2;r_1,s_1)$  with\/ $(r_1, s_1)\in
J_1=\{(3,0), (0,4)\}$;

\item
[$(3)$] a $(C_4, K_{1,3})$-URD$(2K_g;r_2,s_2)$,\/ with\/ $(r_2,
s_2)\in J_2$, where $(r_2, s_2)\in J_2$ and $J_2=\{(r_2, s_2):$
there exists a $(C_4,K_{1,3})-URD(2K_g;r_2,s_2)\}$;
\end{itemize}
then there exists a\/ $(C_4, K_{1,3})$-URD$(2K_v;r,s)$  for each\/
$(r,s)\in J_2+ t\ast J_1$, where\/ $t=\frac{g(u-1)}{3}$ is the
number of parallel classes of the\/ $4$-RGDD of type\/ $g^{u}$ .

\end{lemma}

\section{Small cases}

\begin{lemma}
\label{lemmaA1} $J(2K_4;C_4,K_{1,3})$ = $\{(3,0), (4,0)\}$.
\end{lemma}
\begin{proof}
Let $V(K_{4})$=$\mathbb{Z}_{4}$.

\begin{itemize}
\item $(3,0)$

 $3$ classes of $4$-cycles are: $\{(0,1,2,3)\}$, $\{(0,2,3,1)\}$, $\{(0,2,1,3)\}$.

 \item $(0,4)$

 $4$ classes of $3$-stars can be obtained from the base block $\{\{0;1,2,3\}\}$.
\end{itemize}
\end{proof}

\begin{lemma}
\label{lemmaA3} $J(2K_8;C_4,K_{1,3})$ = $\{(7,0), (4,4), (1,8)\}$.
\end{lemma}
\begin{proof}
Let $V(K_{8})$=$\mathbb{Z}_{8}$.

\begin{itemize}
\item $(7,0), (4,4)$

Take a   $(C_4, K_{1,3})$-URGDD$(2;4,0)$ of type $4^2$ \cite{DQS}
and replace  each  group of size $4$ with the same $(C_4,
K_{1,3})$-URD$(2K_4;r,s)$, with $(r,s)\in \{(3,0), (0,4)\}$ which
exists by Lemma \ref{lemmaA1}.

\item $(1,8)$

$8$ classes of $3$-stars and one class of $4$-cycles are:\\
$\{(0;2,4,6),(1;3,5,7)\}$, $\{(2;4,1,6),(3;5,0,7)\}$, $\{(5;2,0,7),(4;1,3,6)\}$,\\
$\{(0;2,1,3),(4;6,5,7)\}$, $\{(2;4,1,3),(6;5,0,7)\}$, $\{(5;2,0,7),(1;4,3,6)\}$,\\
$\{(1,5,4,0), (2,6,7,3)\}$.

\end{itemize}
\end{proof}

\begin{lemma}
\label{lemmaA2} $J(2K_{12};C_4,K_{1,3})$ = $\{(11,0), (8,4), (5,8),
(2,12)\}$.
\end{lemma}

\begin{proof}
Let $V(K_{12})$=$\mathbb{Z}_{12}$.

\begin{itemize}
\item $(11,0), (8,4)$

Take a   $(C_4, K_{1,3})$-URGDD$(2;8,0)$ of type $4^3$ \cite{DQS}
and replace  each  group of size 4 with the same $(C_4,
K_{1,3})$-URD$(2K_4;r,s)$, with $(r,s)\in \{(3,0), (0,4)\}$, which
exists by Lemma \ref{lemmaA1}.

\item $(5,8)$

$5$ classes of $4$-cycles and $8$ classes of $3$-stars are:\\
$\{(0,1,4,7), (2,3,6,5), (8,11,9,10)\}$, $\{(0,11,10,3), (1,2,9,8), (4,6,7,5)\}$,\\
$\{(3,1,4,8), (2,0,6,10), (7,11,9,5)\}$,$\{(3,11,5,0), (1,2,9,7), (4,6,8,10)\}$,\\
$\{(1,3,2,0), (4,8,11,7), (6,10,9,5)\}$,\\
$\{(0;4,5,6)$, $(7;8,9,10)$, $(11;1,2,3)$\}$, $\{ $(1;5,6,7)$, $(4;9,10,11)$, $(8;0,2,3)\}$,\\
$\{(2;4,6,7)$, $(5;8,10,11)$, $(9;0,1,3)$\}$, $\{$(3;4,5,7)$, $(6;8,9,11)$, $(10;0,1,2)\}$,\\
$\{(3;4,10,6)$, $(8;7,9,5)$, $(11;1,2,0)$\}$, $\{ $(1;10,6,8)$, $(4;9,5,11)$, $(7;0,2,3)\}$,\\
$\{(2;4,6,8)$, $(10;7,5,11)$, $(9;0,1,3)$\}$, $\{$(0;4,10,8)$,
$(6;7,9,11)$, $(5;3,1,2)\}$.

\item $(2,12)$

$2$ classes of $4$-cycles are:\\
$\{(0,5,6,1), (2,7,8,3), (4,9,10,11)  \}$, $\{(0,7,6,11), (1,8,9,2),
( 3,4,5,10) \}$. $12$ classes of $3$-stars can be obtained from the
base blocks:

$\{(4;10,1,6),  (9;2,5,7), (11;3,8,0)\}$.

\end{itemize}
\end{proof}

\begin{lemma}
\label{lemmaA4} There exists a $(C_4, K_{1,3})$-URGDD$(2;r,s)$ of
type $12^{2}$ with $(r,s)\in \{(12,0), (6,8), (0,16)\}$.
\end{lemma}

\begin{proof} Take the groups to be $\{a_1,a_2\ldots, a_{12}\}$ and
$\{b_1,b_2,\ldots, b_{12}\}$.

\begin{itemize}

\item The case $(12,0)$ follows by \cite{DQS}.

\item $(0,16)$

$12$ parallel classes of $4$-cycles are obtained by considering
 $i=1,4,7, 10$ as listed below:\\
$\{(a_i;b_i,b_{i+1},b_{i+2}), (a_{i+1};b_{i+3},b_{i+4},b_{i+5})
,(a_{i+2};b_{i+6},b_{i+7},b_{i+8})$,\\
$(b_{i+9};a_{i+3},a_{i+4},a_{i+5}) ,
(b_{i+10};a_{i+6},a_{i+7},a_{i+8}), (b_{i+11};a_{i+9},a_{i+10},a_{i+11}) \}$,\\
$\{(a_i; b_{i+3},b_{i+4},b_{i+5}), (a_{i+1};b_{i+6},b_{i+7},b_{i+8})
,(a_{i+2};b_{i+9},b_{i+10},b_{i+11}),\\
(b_{i};a_{i+3},a_{i+4},a_{i+5}), (b_{i+1};a_{i+6},a_{i+7},a_{i+8}), (b_{i+2};a_{i+9},a_{i+10},a_{i+11}) \}$,\\
$\{(a_i; b_{i+6},b_{i+7},b_{i+8}), (a_{i+1};b_{i+9},b_{i+10},b_{i+11}) ,(a_{i+2};b_{i},b_{i+1},b_{i+2}), \\
(b_{i+3};a_{i+3},a_{i+4},a_{i+5}), (b_{i+4};a_{i+6},a_{i+7},a_{i+8}), (b_{i+5};a_{i+9},a_{i+10},a_{i+11}) \}$,\\
$\{(a_i; b_{i+9},b_{i+10},b_{i+11}), (a_{i+1};b_{i},b_{i+1},b_{i+2})
,(a_{i+2};b_{i+3},b_{i+4},b_{i+5}),\\
(b_{i+6};a_{i+3},a_{i+4},a_{i+5}), (b_{i+7};a_{i+6},a_{i+7},a_{i+8}), (b_{i+8};a_{i+9},a_{i+10},a_{i+11}) \}$.\\

\item $(6,8)$

$6$ parallel classes of $4$-cycles are:\\
$\{(a_{12},b_{12},a_{5},b_1),(a_{1},b_2,a_{7},b_3),(a_{2},b_4,a_{8},b_5),(a_{3},b_7,a_{4},b_8)$,\\
$(a_{6},b_9,a_{9},b_{10}),(a_{10},b_6,a_{11},b_{11})\}$\\
$\{(a_{12},b_2,a_{8},b_3),(a_{1},b_4,a_{3},b_5),(a_{2},b_6,a_{11},b_7),(a_{4},b_8,a_{6},b_{11})$,\\
$(a_{5},b_9,a_{7},b_{10}),(a_{9},b_{12},a_{10},b_1)\}$\\
$\{(a_{12},b_2,a_{2},b_7),(a_{1},b_4,a_{9},b_6),(a_{3},b_5,a_{4},b_9),(a_{5},b_{12},a_{7},b_{10})$,\\
$(a_{6},b_8,a_{8},b_{11}),(a_{10},b_1,a_{11},b_3)\}$\\
$\{(a_{12},b_4,a_{10},b_5),(a_{1},b_6,a_{6},b_7),(a_{2},b_9,a_{4},b_{10}),(a_{3},b_3,a_{11},b_8)$,\\
$(a_{5},b_2,a_{9},b_{11}),(a_{7},b_{12},a_{8},b_1)\}$\\
$\{(a_{12},b_4,a_{11},b_5),(a_{1},b_7,a_{5},b_9),(a_{2},b_8,a_{7},b_{11}),(a_{3},b_{12},a_{4},b_6)$,\\
$(a_{6},b_1,a_{8},b_{10}),(a_{9},b_2,a_{10},b_3)\}$\\
$\{(a_{12},b_6,a_{2},b_9),(a_{1},b_5,a_{5},b_8),(a_{3},b_7,a_{4},b_{10}),(a_{6},b_{12},a_{8},b_3)$,\\
$(a_{7},b_1,a_{9},b_{11}),(a_{10},b_2,a_{11},b_4)\}$, and $8$
parallel classes of $3$-stars are the last $8$ parallel classes of
the solution for $(0,16)$.

\end{itemize}

\end{proof}

\begin{lemma}
\label{lemmaA5} $J(24,2;C_4,K_{1,3})$ = $I(24)$.
\end{lemma}

\begin{proof}

Take a   $(C_4, K_{1,3})$-URGDD$(2; r,s)$ of type $12^2$ and index 2
with $(r,s)\in \{(12,0), (0,16)\}$ which exists by Lemma
\ref{lemmaA4}. Replace each group of size 12 with the same $(C_4,
K_{1,3})$-URD$(2K_{12};r,s)$, where $(r,s)\in \{(11,0), (8,4),
(5,8), (2,12)\}$ which exists by Lemma \ref{lemmaA2}.

\end{proof}
\begin{lemma}
\label{lemmaA6} $J(36,2;C_4,K_{1,3})$ = $I(36)$.
\end{lemma}

\begin{proof}

Take a   $(C_4, K_{1,3})$-URGDD$(2; r,s)$ of type $12^3$ and index 2
with $(r,s)\in \{(24,0), (0,32)\}$ which exists by Lemma
\ref{lemmaA4}. Replace each group of size $12$ with the same $(C_4,
K_{1,3})$-URD$(2K_{12};r,s)$, where $(r,s)\in \{(11,0), (8,4),
(5,8), (2,12)\}$ which exists by Lemma \ref{lemmaA2}.

\end{proof}

\begin{lemma}
\label{lemmaA7}  There exists a \/ $(C_4,
K_{1,3})$-IURD$(2K_{20}-2K_8;[r,s],[\bar{r},\bar{s}])$ with
$(r,s)\in \{(7,0), (4,4), (1,8)\}$ and $(\bar{r},\bar{s})\in
\{(12,0), (9,4), (6,8), (3,12), (0,16)\}$.
\end{lemma}

\begin{proof}
Let the point set of $K_{20}$ be $\mathbb{Z}_{20}$ and the point set
$\{0,1,\ldots,7\}$ be the hole. The following resolution classes
$(7,0), (4,4), (1,8)$ cover the same edges on the point set
$K_{20}-K_8$.

\begin{itemize}

\item The parallel classes of resolution $(7,0)$:\\
$\{(8,9,11,10),(12,13,15,14),(16,17,19,18)\}$,\\
$\{(8,11,15,12),(9,10,19,16),(13,14,18,17)\}$,\\
$\{(8,13,10,15),(9,18,11,19),(12,16,14,17)\}$,\\
$\{(8,14,16,10),(12,18,15,9),(11,17,19,13)\}$,\\
$\{(8,16,15,12),(14,10,19,11),(18,9,13,17)\}$,\\
$\{(8,18,10,15),(14,13,16,19),(12,11,9,17)\}$,\\
$\{(8,9,15,14),(10,11,17,16),(12,13,19,18)\}$.

\item The parallel classes of resolution $(4,4)$ are the last $4$ parallel classes in resolution $(7,0)$
above and the following $4$ parallel classes of $3$-stars:\\
$\{(8;9,10,11), (14;12,13,15), (19;16,17,18)\}$, \\
$\{(9;10,11,19), (15;8,12,13), (17;14,16,18)\}$,\\
$\{(10;11,15,19), (12;8,13,17), (18;9,14,16)\}$, \\
$\{(11;15,18,19), (13;8,10,17), (16;9,12,14)\}$.

\item The parallel classes of resolution $(1,8)$ are the last parallel class in resolution $(7,0)$
above and the following $8$ parallel classes of $3$-stars:\\
$\{(8;9,10,11),(12;13,14,15),(16;17,18,19)\}$,\\
$\{(8;10,12,13),(9;11,15,16),(14;17,18,19)\}$,\\
$\{(9;10,11,12),(13;14,17,19),(15;8,16,18)\}$,\\
$\{(10;11,13,19),(14;8,15,16), (17;9,12,18)\}$,\\
$\{(10;14,15,16),(17;11,12,13),(18;8,9,19)\}$,\\
$\{(11;12,14,18),(15;8,10,13),(19;9,16,17)\}$,\\
$\{(11;13,15,19),(16;8,12,14),(18;9,10,17)\}$,\\
$\{(12;8,15,18),(13;9,14,16),(19;10,11,17)\}$.

\end{itemize}

Now all the edges that are not covered in the above resolution
classes on the point set $K_{20}-K_8$ will be covered by the
following full parallel classes on the point set $K_{20}$.

\begin{itemize}

\item The parallel classes of resolution $(12,0)$:\\

$\{(0,8,1,9),(2,10,3,11),(4,12,5,13),(6,14,18,16),(7,17,15,19)\}$,\\
$\{(0,8,1,9),(2,10,3,11),(4,12,5,13),(6,16,7,18),(14,17,15,19)\}$,\\
$\{(0,10,1,11),(2,8,3,9),(4,14,12,16),(5,15,13,18),(6,17,7,19)\}$,\\
$\{(0,10,1 11),(2,8,3,9),(4,17,5,19),(6,12,7,14),(13,16,15,18)\}$,\\
$\{(0,12,1,13),(2,14,3,15),(4,8,16,9),(5,17,6,19),(7,10,18,11)\}$,\\
$\{(0,12,1,13),(2,14,11,16),(3,15,4,17),(5,8,19,9),(6,10,7,18)\}$,\\
$\{(0,14,1,15),(2,18,3,19),(4,9,7,16),(5,8,17,10),(6,12,11,13)\}$,\\
$\{(0,14,1,15),(2,13,8,17),(3,16,5,18),(4,10,12,19),(6,9,7,11)\}$,\\
$\{(0,16,1,17),(2,15,4,18),(3,13,10,14),(5,9,6,11),(7,8,19,12)\}$,\\
$\{(0,16,1,17),(2,12,3,19),(4,11,8,18),(5,10,9,14),(6,13,7,15)\}$,\\
$\{(0,18,1,19),(2,13,3,16),(4,11,5,14),(6,8,7,15),(9,12,10,17)\}$,\\
$\{(0,18,1,19),(2,12,3,17),(4,8,6,10),(5,15,11,16),(7,13,9,14)\}$.

\item The parallel classes of resolution $(9,4)$ are the following $9$ parallel classes of $4$-cycles:\\
$C_1=\{(0,10,1,11),(2,8,3,9),(4,17,5,19),(6,12,7,14),(13,16,15,18)\}$,\\
$C_2=\{(0,12,1,13),(2,14,3,15),(4,8,16,9),(5,17,6,19),(7,10,18,11)\}$,\\
$C_3=\{(0,14,1,15),(2,12,11,16),(3,13,6,18),(4,17,8,19),(5,9,7,10)\}$,\\
$C_4=\{(0,10,1,11),(2,8,3,9),(4,14,12,16),(5,15,13,18),(6,17,7,19)\}$,\\
$C_5=\{(0,18,1,19),(2,13,3,17),(4,8,11,16),(5,10,6,15),(7,12,9,14)\}$\\
$C_6=\{(0,8,1,9),(2,10,3,11),(4,12,5,13),(6,14,18,16),(7,17,15,19)\}$,\\
$C_7= \{(0,8,1,9),(2,10,3,11),(4,12,5,13),(6,16,7,18),(14,17,15,19)\}$,\\
$C_8=\{(0,12,1,13),(2,14,3,15),(4,10,6,11),(5,16,7,18),(8,17,9,19)\}$,\\
$C_9=\{(0,14,4,18),(1,16,2,19),(3,12,10,17),(5,8,13,11),(6,9,7,15)\}$,\\
and the following $4$ parallel classes of $3$-stars:\\
$S_1=\{(0;15,16,17),(7;8,11,13),(12;6,10,19),(14;1,5,9),(18;2,3,4)\}$,\\
$S_2=\{(1;16,17,18),(4;9,10,15),(5;8,11,14),(13;2,6,7),(19;0,3 12)\}$,\\
$S_3=\{(2;12,18,19),(6;8,9,11),(10;13,14,17),(15;1,4,7),(16;0,3,5)\}$,\\
$S_4=\{(3;12,16,19),(8;6,7,18),(9;5,10,13),(11;4,14,15),(17;0,1,2)\}$.

\item The $6$ parallel classes of $4$-cycles of resolution $(6,8)$ are $C_1,C_2,C_3,C_4,C_5$ in the resolution $(9,4)$ above together with the following parallel class of $4$-cycles:\\
$\{(0,8,17,9),(1,12,4,13),(2,10,6,11),(3,14,19,15),(5,16,7,18)\}$
and the $8$ parallel classes of $3$-stars are $S_1,S_2,S_3,S_4$ in the resolution $(9,4)$ above together with the following $4$ parallel classes of $3$-stars:\\
$S_5=\{(0;9,14,18),(5;11,12,13),(10;3,4,17),(16;1,2,6),(19;7,8,15)\}$,\\
$S_6=\{(1;8,9,19),(7;15,16,18),(11;2,3,13),(12;0,5,10),(14;4,6,17)\}$,\\
$S_7=\{(2;10,14,19),(4;11,12,13),(6;9,16,18),(8;0,1,5),(17;3,7,15)\}$,\\
$S_8=\{(3;10,11,12),(9;1,7,19),(13;0,5,8),(15;2,6,17),(18;4,14,16)\}$.

\item The $3$ parallel classes of $4$-cycles of resolution $(3,12)$ are $C_1,C_2,C_3$ in the resolution $(9,4)$ and the $12$ parallel classes of $3$-stars are $S_1,S_2,S_3,S_4$ in
the resolution $(9,4)$, $S_5,S_6,S_7,S_8$ in
the resolution $(6,8)$ above together with the following $4$ parallel classes of $3$-stars:\\
$S_9=\{(0;9,10,18),(1;11,12,13),(4;8,14,16),(15;3,5,19),(17;2,6,7)\}$,\\
$S_{10}=\{(2;9,10,11),(3;8,13,17),(5;15,16,18),(12;4,7,14),(19;0,1,6)\}$,\\
$S_{11}=\{(6;10,11,15),(8;0,2,17),(14;3,9,19),(16;4,7,12),(18;1,5,13)\}$,\\
$S_{12}=\{(7;14,18,19),(9;3,12,17),(10;1,5,6),(11;0,8,16),(13;2,4,15)\}$.

\item The $16$ parallel classes of $3$-stars of resolution $(0,16)$ are $S_i$, $i=2,3,...,12$ above together with the following $5$ parallel classes of $3$-stars:\\
$\{(0;10,11,17),(2;8,12,15),(13;7,16,18),(14;1,3,9),(19;4,5,6)\}$,\\
$\{(0;13,15,16),(5;14,17,19),(7;8,9,10),(12;1,6,11),(18;2,3,4)\}$,\\
$\{(1;10,13,14),(12;0,7,19),(16;2,8,9),(17;4,5,6),(18;3,11,15)\}$,\\
$\{(3;8,13,15),(4;9,17,19),(10;5,12,18),(11;1,7,16),(14;0,2,6)\}$,\\
$\{(6;12,13,18),(7;10,11,14),(8;4,17,19),(9;2,3,5),(15;0,1,16)\}$.

\end{itemize}
\end{proof}

\begin{lemma}
\label{lemmaA12} $J(2K_{20};C_4,K_{1,3})$ = $I(20)$.
\end{lemma}
\begin{proof}
Replace the hole of size $8$ in Lemma \ref{lemmaA7} by a $(C_4,
K_{1,3})$-URD$(2K_{8};r,s)$, with $(r,s)\in \{(7,0), (4,4), (1,8)\}$
which exists by Lemma \ref{lemmaA3}.
\end{proof}

\section{Main results}

\begin{lemma}
\label{lemmaC1} For every\/ $v\equiv 0\pmod{12}$ $I(v)\subseteq J(2K_v; C_4, K_{1,3})$.

\end{lemma}

\begin{proof}

For $v=12,24,36$ the conclusion follows from Lemmas \ref{lemmaA2},
\ref{lemmaA5} and \ref{lemmaA6}. For $v\geq 48$ start with a
$4$-RGDD $G$ of type $12^{\frac{v}{12}}$ \cite{CD} and apply Lemma
\ref{lemmaP2} with $g=12, u= \frac{v}{12}$ and $t=\frac{(v-12)}{3}$
(The input designs are  a $(C_4, K_{1,3})$-URD$(2K_4;r,s)$, with
$(r,s)\in \{(3,0),(0,4)\}$, which exists by Lemma \ref{lemmaA1} and
a $(C_4, K_{1,3})$-URD$(2K_{12};r,s)$,  with $(r,s)\in\{(11,0),
(8,4), (5,8), (2,12)\}$, which exists by Lemma \ref{lemmaA2}). This
implies
$$J(2K_v; C_4, K_{1,3}) \supseteq \{\{(11,0), (8,4), (5,8), (2,12)\}+ \frac{(v-12)}{3}\ast\{(3,0),(0,4)\}\}.$$
Since $\frac{v-12}{3}\ast  \{(3,0),(0,4)\}$=$\{(v-12-3x, 4x),
x=0,\ldots,\frac{v-12}{3}\}$, it is easy to see that $\{\{(11,0),
(8,4), (5,8), (2,12)\}+ \frac{(v-12)}{6}\ast\{(3,0),
(0,4)\}\}$=$I(v)$. This completes the proof.

\end{proof}

\begin{lemma}
\label{lemmaC2} For every\/ $v\equiv 8\pmod{24}$ $I(v)\subseteq J(2K_v; C_4, K_{1,3})$.

\end{lemma}

\begin{proof}
For $v=8$ the result follows by Lemma \ref{lemmaA3}. For $v\geq32$,
start with a $4$-RGDD  of type $8^{\frac{v}{8}}$  \cite{CD} and
apply Lemma \ref{lemmaP2} with $g=8, u= \frac{v}{8}$ and
$t=\frac{(v-8)}{3}$ (The input designs are  a $(C_4,
K_{1,3})$-URD$(2K_4;r,s)$, with $(r,s)\in \{(3,0),(0,4)\}$, which
exists by Lemma \ref{lemmaA1} and a $(C_4,
K_{1,3})$-URD$(2K_8;r,s)$, with $(r,s)\in\{(7,0), (4,4), (1,8)\}$,
which exists by Lemma \ref{lemmaA3}). Proceeding as in Lemma
\ref{lemmaC1} the result follows.
\end{proof}

\begin{lemma}
\label{lemmaC3} For every\/ $v\equiv 4\pmod{12}$,$I(v)\subseteq J(2K_v; C_4, K_{1,3})$.
\end{lemma}

\begin{proof}
For $v=4$ the result follows by Lemma \ref{lemmaA1}. For $v\geq16$,
start with a $4$-RGDD  of type $4^{\frac{v}{4}}$  \cite{CD} and
apply Lemma \ref{lemmaP2} with $g=4, u= \frac{v}{4}$ and
$t=\frac{v-4}{3}$ (The input design is  a $(C_4,
K_{1,3})$-URD$(2K_4;r,s)$, with $(r,s)\in \{(3,0),(0,4)\}$, which
exists by Lemma \ref{lemmaA1}). Proceeding as in Lemma \ref{lemmaC1}
the result follows.

\end{proof}

\begin{lemma}
\label{lemmaC4} For every\/ $v\equiv 20\pmod{24}$, $I(v)\subseteq
J(2K_v; C_4, K_{1,3})$.
\end{lemma}

\begin{proof}
The case $v=20$ follows by Lemma \ref{lemmaA12}. For $v\geq44$
start from a   $2$-frame $\cF$ of type $1^{\frac{v-8}{12}}$ with
groups $G_i$, $i=1,2,\ldots,\frac{v-8}{12}$ \cite{CD}. Then expand
each point by $12$ points and add a set $H=\{a_1,a_2,\ldots,a_8\}$.
For $i=1,2,\ldots,\frac{v-8}{12}$,  let $P_{i}$ be the partial
factor which miss the group $G_i$.

Replace  each block $b\in P_{i}$ by a $(C_4,K_{1,3})$-URGDD$(2;
r_1,s_1)$ of type $12^2$ and index 2, say $D_{i}^b$  on the vetex
set of $b\times \{1,2,...,12\}$ with $(r_1,s_1)\in \{(12,0), (6,8),
(0,16)\}$, which exists by Lemma \ref{lemmaA4}.

For $i=1,2,\ldots,\frac{v-8}{12}$ place on $H\cup
(G_i\times\{1,2,\ldots,12\})$ a copy of a $(C_4,
K_{1,3})$-IURD$(2K_{20}-2K_8;[x_1,y_1],[x,y])$, say $D_i$  with
$(x_1,y_1)\in \{(7,0), (4,4), (1,8)\}$ and $(x,y)\in \{(12,0),
 (6,8), (0,16)\}$, which exists by Lemma \ref{lemmaA7}.
Combine the parallel classes of $D_{i}^b$ with the full classes of
$D_i$ so to obtain $r_2$ classes of $C_4$ and $s_2$ classes of
$K_{1,3}$ with $(r_2,s_2)\in \{(\frac{v-8}{12})\ast \{(12,0), (6,8),
(0,16)\}\}$.

Fill the hole $H$ with a copy of a $(C_4, K_{1,3})$-URD$(2K_8;r,s)$
say $D$ with $(r_4,s_4)\in\{(7,0), (4,4), (1,8)\}$, which exists by
Lemma \ref{lemmaA3}. Combine the classes of $D$ with the partial of
$D_i$ so to obtain $r_4$ classes of $C_4$ and $s_4$ classes of
$K_{1,3}$ with $(r_4,s_4)\in\{(7,0), (4,4), (1,8)\}$.

This gives a $(C_4, K_{1,3})$-URD$(2K_v;r,s)$, with $(r,s)\in
\{\{(7,0), (4,4), (1,8)\}+ \frac{((v-8)}{12}\ast\{(12,0), (6,8),
(0,16)\}\}$. Proceeding as in Lemma \ref{lemmaC1} we obtain the
result.

\end{proof}

Combining Lemmas   \ref{lemmaC1}, \ref{lemmaC2}, \ref{lemmaC3} and
\ref{lemmaC4} we obtain the main theorem of this article.

\begin{thm}
For each\/  $v\equiv 0\pmod{4}$, we have $J(2K_v; C_4,
K_{1,3})$=$I(v)$.
\end{thm}

\end{document}